\documentclass{article}
\usepackage[T1]{fontenc}
\usepackage{a4wide}
\usepackage{authblk}
\usepackage{graphicx} 
\usepackage{systeme}
\usepackage{amsmath}
\usepackage{amssymb}
\usepackage{hyperref}
\usepackage{amsthm}
\usepackage{xcolor}
\newtheorem{thm}{Theorem}
\newtheorem{assum}[thm]{Assumption}
\newtheorem{prop}[thm]{Proposition}
\newtheorem{defn}[thm]{Definition}

\newtheorem{rem}[thm]{Remark}

\newtheorem{lem}[thm]{Lemma}

\DeclareMathAlphabet{\mathbbb}{U}{bbold}{m}{n}
\renewcommand{\leq}{\leqslant}
\renewcommand{\le}{\leqslant}
\renewcommand{\geq}{\geqslant}
\renewcommand{\ge}{\geqslant}

\renewcommand{\Im}{\mathrm{Im}}
\renewcommand{\Re}{\mathrm{Re}}

\newcommand{\fonction}[5]{
\begin{align*}
\displaystyle
\begin{array}{lrcl}
#1: & #2 & \longrightarrow & #3 \\
    & #4 & \longmapsto & #5
\end{array}
\end{align*}}

\newcommand{\R}{\mathbb{R}}
\newcommand{\N}{\mathbb{N}}
\newcommand{\C}{\mathbb{C}}

\title{A Spectral Exponential Stability Criterion for Integral Difference Equations and Delay Differential Equations in various state spaces\thanks{This project received funding from the Agence Nationale de la Recherche via grant PANOPLY ANR-23-CE48-0001-01.}}
\author[1]{Adam Braun}
\author[1]{Jean Auriol}
\author[1]{Lucas Brivadis}
\affil[1]{Université Paris-Saclay, CNRS, CentraleSupélec, Laboratoire des Signaux et Systèmes, 91190, Gif-sur-Yvette, France. Emails:
        {\tt\small adam.braun@centralesupelec.fr; jean.auriol@centralesupelec.fr; lucas.brivadis@centralesupelec.fr}}%

\date{}
\begin{document}

\maketitle

\begin{abstract}
It is well-known that the exponential stability of Integral Difference Equations and Delay Difference Equations, in the usual state space of continuous functions, is equivalent to the location of the roots of its associated characteristic equation strictly in the open left half-plane 
(see e.g. \cite[Chapter 9]{halebook}). In this paper, we
use results from \cite[Chapter~4]{GripenbergLondenStaffans1990} to
show that this characterization still holds for other functional state spaces: Lebesgue spaces, the space of Borel measurable bounded functions, and the space of functions with bounded variation.
\end{abstract}

\section{Introduction}

Integral Difference Equations (IDE) and Delay Difference Equations (DDE) constitute a class of functional equations incorporating both pointwise and distributed delays in the state. They are widely used in the modeling of engineering and biological systems involving transport and measurement delays~\cite{Niculescu2001DelayEffects}, among many other applications like sampled-data systems~\cite{Fridman2014TimeDelay}, or epidemic models~\cite{COOKE197687}.
Often regarded as a subclass of neutral time-delay systems~\cite{halebook}, IDEs have received little attention in the literature comparatively to DDEs (well studied in~\cite{halebook} or~\cite{Niculescu2001DelayEffects} for example). In~\cite{Chitour_2023}, controllability criteria were established for IDEs with pointwise delays, while~\cite{fueyo2025lqapproximatecontrollabilityfrequency} extended these results to systems involving both pointwise and distributed delays. The Hale–Silkowski criterion~\cite{halebook}, originally formulated for IDEs with pointwise delays, was later generalized in~\cite{felipe} to encompass more general IDEs including distributed terms. Lyapunov-based stability analyses were proposed in~\cite{MONDIE2025100985} for distributed delays only, and necessary Lyapunov conditions for stability were further derived in~\cite{DAMAK20143299,MELCHORAGUILAR2023105536}. More recently,~\cite{LAMARQUE2025112437} established necessary and sufficient conditions for input-to-state stability (ISS) in inhomogeneous IDEs combining both pointwise and distributed delays.
In recent years, this class of functional equations has attracted renewed interest due to its applications in the analysis of partial differential equations (PDEs). In particular, interconnected one-dimensional linear hyperbolic balance laws~\cite{bastin_coron} and IDEs share equivalent stability properties~\cite{auriol_hdr}. This equivalence enables the design of control laws for PDEs within the IDE framework like in~\cite{braun2025stabilizationchainhyperbolicpdes} or~\cite{Braun_2inputs}.

A fundamental exponential stability criterion for delay difference equations (DDEs) was first established in~\cite{halebook}, relating the exponential stability of a DDE to the location of the roots of its associated characteristic equation. In \cite[Chapter 9]{halebook}, this criterion was adapted to IDEs, viewed as a particular class of neutral DDEs. This adaptation has since proven effective in the practical design of stabilizing control laws, as illustrated in~\cite{auriol2024stabilizationintegraldelayequations}.
To the best of our knowledge, this criterion has been established only in the state space of continuous functions, as shown in~\cite[Chapter~9]{halebook}, or in the absence of distributed delays, as in~\cite{hale_feedback_2}. 
Although the methodologies presented in~\cite{Lunel_Kaashoek_characteristic} and~\cite{HENRY1974106} could likely be adapted to prove the criterion in more general state spaces, such a rigorous extension has not yet been carried out. 
Establishing the criterion beyond the space of continuous functions is essential, since Lebesgue spaces naturally arise, for example, in the equivalence between PDE and IDE stabilizability frameworks~\cite{Jean_Adam_Ouidir_Mathieu}.

The purpose of this note is to extend this exponential stability criterion to IDEs and DDEs  where the state space is a Lebesgue space, the space of functions with bounded variation , or the space of Borel measurable bounded functions (this space is considered only for IDEs). Our proof relies on several results from~\cite[Chapter~4]{GripenbergLondenStaffans1990}, after verifying that the assumptions required for their application are satisfied. 
The paper is organized as follows. Section~\ref{section:prelim} presents preliminary results, including the well-posedness of the general IDE under consideration. Section~\ref{main result} is devoted to the proof of the exponential stability criterion for a general IDE. In Section~\ref{critere_DDE} we use the same methodology to prove an exponential stability criterion for a general DDE.
\section{Mathematical setting and notation}

The notation used in this manuscript are borrowed from
%
\cite{GripenbergLondenStaffans1990}.
Let $n>0$ and $p\geq 1$ be integers. The set of positive integers is denoted $\N^*$.
 The $n\times n$ identity matrix is denoted by $I$.
The euclidean norm on $\mathbb{R}^n$ is denoted by $|\cdot|$.
The indicator function of the set  $C\subset \R$ is
denoted by $\mathbbb{1}_{C}$.
In the whole article, $\eta$ is a weight function (see~\cite[Definition 2.1, Chapter 4]{GripenbergLondenStaffans1990}, typically $\eta(s)=e^{\nu s}$ for some $\nu\geq0$ in this article).
 For $C\subset \R$ with nonempty interior, $V(C, \R^n)$ is a Banach space from the list:
    \begin{enumerate}
        \item $L^p(C, \R^n)$ for some $p\in [1,\infty]$, the usual Lebesgue space;
        \item $B^\infty(C, \R^n)$ the borel measurable bounded functions with the sup norm;
        \item $BV(C, \R^n)$ the functions of bounded variations; the norm is the sum of the total variation and the sup-norm.
    \end{enumerate}
We denote $\|\cdot\|_V$ the associated norm.
For $X \in V(C, \R^n)$,  for all $t\in C$, $X_t$ denotes the partial trajectory associated to $X$ and is defined for all $s\in C$ such that $t+s \in C$, by $X_t(s):= X(t+s)$.
 Furthermore, $V_{loc}(C, \R^n)$ denotes the functions that are in $V(K, \R^n)$ for every compact $K\subset C$.
   For $C\subset \R$ with nonempty interior, $V(C, \eta, \R^n):=\{f: C \to \R^n,~ \eta f\in V(C, \R^n)\}$ is a weighted space with norm $\|f\|_{V_\eta}:= \|\eta f\|_V$.
   
   For $C\subset \mathbb{R}$ with nonempty interior, $C_b(C,\C^{n\times n})$ is the space of bounded continuous matrix valued application from $K$ to $\C^{n\times n}$ with the sup norm $\|\cdot\|_{\infty}$.
   
  We denote by $M(C, \mathbb{R}^{n\times n})$, the space of finite measure on $C\subset \mathbb{R}$ (with nonempty interior) equipped with the bounded total variation norm  $\|l\|_{M(K)} = \int_K |l(ds)|<\infty$.
     We denote by $M_{\textit{loc}}(\mathbb{R}_+, \mathbb{R}^{n\times n})$, the space of measures $\mu \in M(K,\mathbb{R}^{n\times n})$ for every compact subset $K$ of $\mathbb{R^+}$; and by $M(\mathbb{R}_+, \eta, \mathbb{R}^{n\times n})$  the set of measures $l$ such that $l\eta\in M(\mathbb{R}_+, \R^{n\times n})$ 
     (see~\cite[Section 4.3]{GripenbergLondenStaffans1990}) with the norm $\|l\|_{M(\mathbb{R}_+,\eta)} = \int_0^{\infty}|l(ds)|\eta(s)ds$. 
  For $a\in \mathbb{R}$, $\delta_{a} \in M(\mathbb{R}, \mathbb{R}^{n\times n})$ is the Dirac distribution at $a$.
    
 For a matrix $M\in \mathbb{R}^{n\times n}$, $|M|$ is a matrix norm.
    For $\mu \in M_{loc}(\mathbb{R}_+, \mathbb{R}^{n\times n})$, $\hat{\mu}$ denotes its Laplace transform (see~\cite[Definition 2.2, Chapter 3]{GripenbergLondenStaffans1990}),
    $$\hat{\mu}(z) = \int_0^\infty \mu(ds)e^{-zs}.$$It is always well defined on $\C$ if $\mu$ has compact support.
    For $\mu, m \in M(\R, \R^{n\times n})$, their convolution is denoted by $\mu*m$. It is defined as the completion of the measure that to each Borel set $E\subset \R$ assigns the value
    $$(\mu*m)(E) = \int_\R \mu(ds)m(E-s),$$
    where $E-s:= \{t-s:t\in E\}$, (see e.g.~\cite[Definition 1.1, Section 4.1]{GripenbergLondenStaffans1990}).
    For $C\subset \mathbb{R}$ with nonempty interior, and $p\in [1,\infty]$, $L^p(C, \mathbb{R}^{n\times n}) := 
\{ M : C \to \mathbb{R}^{n\times n} \ \big|\ 
\|M\|_{L^p}^p := \int_C |M(t)|^p \, dt < \infty \}$ and
 $L^p(C, \eta, \mathbb{R}^{n\times n}) := 
\{ M : C \to \mathbb{R}^{n\times n} \ \big|\ 
\|M\|_{L^p_\eta}^p := \int_C |M(t)|^p \eta^p \, dt < \infty \}$.
For a closed set $Z \subset \C$, $\operatorname{dist}(\cdot,Z)$ denotes the distance to $Z$.

\section{IDE of interest and preliminary results}
\label{section:prelim}
In this section, we introduce the general IDE model that serves as the foundation for our analysis and examine the conditions ensuring its well-posedness. We conclude the section by recalling the notion of exponential stability associated with this class of equations. Let us consider the following IDE, for any $X_0$ in $V([-\tau_*, 0], \R^n)$, 
\begin{equation}\label{main IDE}
\begin{cases}
 X(t) + \sum_{k=1}^{\infty} A_k X(t-\tau_k) + 
 \int_0^{\tau_*} N(s) X(t-s)\,ds = 0,
 \quad \textit{$t>0$},\\[4pt]
 X(t) = X_0(t), \quad \textit{$-\tau_* \le t \le 0$}.
\end{cases}
\end{equation}
with $\tau_* >0$.
A function $X$ is a solution of~\eqref{main IDE} if, for all $t \ge 0$, we have $X_t \in V([-\tau_*,0],\mathbb{R}^n)$, and the following conditions hold:  
when $V = B^\infty$, equation~\eqref{main IDE} must hold for all $t \ge -\tau_*$;  
when $V = BV$ or $V = L^p$, equation~\eqref{main IDE} is required to hold for almost every $t$.

The following assumption is required to show the well-posedness 
of~\eqref{main IDE}. Its stands for the whole article.
\begin{assum} ~
    \begin{itemize}
        \item The sequence of $n\times n$ matrices $(A_k)_{k\in \mathbb{N}^*}$ is such that  $\sum_{k=1}^\infty  |A_k| <\infty$.
        \item The sequence $(\tau_k)_{k\in \mathbb{N}^*}$ is strictly increasing and for all $k\in \mathbb{N}^*$, $\tau_k \in (0, \tau_*]$.
        \item The matrix-valued kernel $N$ belongs to $L^1((0,\tau_*) , \mathbb{R}^{n\times n})$.
    \end{itemize}
\label{assump:main_ide_well_posedness}\end{assum}
We begin with the following proposition, which reformulates equation~\eqref{main IDE} within the framework developed in~\cite[Chapter 4]{GripenbergLondenStaffans1990}. This reformulation enables a direct application of the well-posedness theory established therein.
\begin{prop} \label{one to one}
For all initial data $X_0 \in V([-\tau_*,0],\mathbb{R}^n)$, if $X$ is a solution of~\eqref{main IDE}, then the function $x(t) := X(t)\,\mathbf{1}_{t\geq 0}(t)$ is a solution of
\begin{equation}
    \label{IDE measure form}
    x(t) + \int_0^t \mu(ds)\, x(t-s) = f(t), \quad t \ge 0,
\end{equation}
where $\mu(ds) = \sum_{k=1}^\infty A_k\, \delta_{\tau_k}(ds) + N(s)\,ds \in M(\mathbb{R}_+, \mathbb{R}^{n\times n})$ is compactly supported in $[0,\tau_*]$, and
\begin{equation}
    \label{def:f_IDE_measure}
    f(t) = - \mathbf{1}_{t\leq \tau_*}(t) \int_t^{\tau_*} \mu(ds)\, X_0(t-s).
\end{equation}
Conversely, if $x$ is a solution to~\eqref{IDE measure form}, then $X(t) = x(t)\,\mathbf{1}_{t> 0}(t) + X_0(t)\,\mathbf{1}_{-\tau_* \le t \le 0}(t)$ 
is a solution to~\eqref{main IDE}. Moreover, equation~\eqref{IDE measure form} admits a unique solution 
\(x \in V_{\mathrm{loc}}(\mathbb{R}_+, \mathbb{R}^n)\), given by
\begin{equation}
    \label{sol explicit}
    x(t) = f(t) - \int_0^t \rho(ds)\, f(t-s), \quad  t \ge 0,
\end{equation}
where $\rho \in M_{\mathrm{loc}}(\mathbb{R}_+, \mathbb{R}^{n\times n})$ is the resolvent of $\mu$, i.e.,
the unique measure satisfying
\[
\rho + \mu * \rho \;=\; \rho + \rho * \mu \;=\; \mu.
\]

\end{prop}
\begin{proof}
   We show the one-to-one correspondence between the solutions of~\eqref{main IDE} and~\eqref{IDE measure form}.
    Let $x$ be a solution of~\eqref{IDE measure form} and $X_0\in V([-\tau_*, 0], \R^n)$. Define $$\Tilde{x}(t) = \begin{cases}
        x(t) &\textit{if}~  t>0\\
        X_0(t)&\textit{if}~  -\tau_*\leq t\leq0.
    \end{cases}$$
    Then, evaluating equation~\eqref{IDE measure form} for $t\geq \tau_*$, we obtain
    $$x(t) + \int_0^{\tau_*} \mu(ds)x(t-s) =0,$$
    i.e
    $$ x(t) + \sum_{k=1}^\infty  A_kx(t-\tau_k) + \int_0^{\tau_*}N(s)x(t-s)ds=0.$$
    Evaluating equation~\eqref{IDE measure form} for $0\leq t \leq \tau_*$, we obtain
    $$x(t)+ \int_0^{t} \mu(ds)x(t-s) + \int_t^{\tau_*}\mu(ds)x_0(t-s)=0, $$
    i.e.
    \begin{align*}x(t) &+ \sum_{k=1}^\infty  A_kx(t-\tau_k)\mathbbb{1}_{t\geq \tau_k} + \sum_{k=1}^\infty  A_k X_0(t-\tau_k)\mathbbb{1}_{t\leq \tau_k}   \\
    &+ \int_0^{t}N(s)x(t-s)ds + \int_{t}^{\tau_*}N(s)X_0(t-s)ds=0.\end{align*}
    Hence, for all $t >0$,
    $$ \Tilde{x}(t) + \sum_{k=1}^\infty  A_k\Tilde{x}(t-\tau_k) + \int_0^{\tau_*}N(s)\Tilde{x}(t-s)ds=0.$$
    Conversely, if \(X\) solves~\eqref{main IDE}, then \(x(t) = X(t)\mathbbb{1}_{t \ge 0}(t)\) satisfies~\eqref{IDE measure form}. The existence and uniqueness of the solution $x\in V_{\textit{loc}}(\mathbb{R}_+, \mathbb{R}^n)$ of~\eqref{IDE measure form}, and its explicit formulation directly follow from~\cite[Theorem 1.7, Chapter 4]{GripenbergLondenStaffans1990} observing that $f \in V(\mathbb{R}_+, \mathbb{R}^n)$ (using~\cite[Theorem 3.5, Section 4.3]{GripenbergLondenStaffans1990}) and $|\det(I+\mu(\{0\}))|>0$. The last inequality is a consequence of the fact that
    $\mu(\{0\}) =0$, i.e $\mu$ has no atomic part at zero,
    which directly follows from Assumption~\ref{assump:main_ide_well_posedness}.
    \end{proof}
We recall here the definition of exponential stability for the IDE~\eqref{main IDE}.
\begin{defn}[Exponential stability]
 The IDE~\eqref{main IDE} is said to be exponentially stable, if there exist $\nu >0$, and $C\geq1$ such that for all $X_0\in V((-\tau_*, 0), \R^n)$, the solution to~\eqref{main IDE} satisfies, for all $t\geq 0$,
    $$\|X_t\|_V\leq Ce^{-\nu t}\|X_0\|_V.$$
\end{defn}
In the next section, we present the main result of the paper, namely a necessary and sufficient condition ensuring the exponential stability of~\eqref{main IDE}.
\section{Exponential stability criterion for the IDE}
\label{main result}
In this section, we derive the exponential stability criterion for the IDE~\eqref{main IDE} by verifying the conditions required to apply~\cite[Corollary~4.7, Section~4.4]{GripenbergLondenStaffans1990}. This criterion characterizes exponential stability in terms of the location of the zeros of the characteristic equation associated with~\eqref{main IDE}. The proof consists in establishing that the determinant of the characteristic matrix remains uniformly bounded away from zero on the region of the complex plane where it does not vanish. For this purpose, we use~\cite[Proposition~A.5]{chitour:hal-05337430}, which provides a convenient lower bound for such determinants. We begin by introducing the characteristic equation corresponding to~\eqref{main IDE}.

\begin{defn}[Characteristic equation]
    
    The characteristic equation associated to the IDE~\eqref{main IDE} is defined for all $z\in \mathbb{C}$, by
    \begin{equation}\label{Characteristic eq}
       \operatorname{det}\Delta(z) = 0,
    \end{equation}
    where 
    $$\Delta(z) :=I+\hat{\mu}(z)= \Delta_0(z) + R(z),$$
    with $$\Delta_0(z):= I +\sum_{k=1}^\infty A_k e^{-\tau_k z},\quad R(z) := \int_0^{\tau_*}N(s)e^{-s z}ds.$$
\end{defn}

We recall the following lemma, that can be found in \cite[Proposition A.5]{chitour:hal-05337430}. It will be used to check the assumptions of \cite[Corollary~4.7, Section~4.4]{GripenbergLondenStaffans1990}.

\begin{lem}{\cite[Proposition A.5]{chitour:hal-05337430}} \label{lem:lewin_perturbed}
For all $\beta >0$, and all $0<\delta <\beta$, there exists $m(\delta)>0$ such that if $z\in \Omega_{\delta}:=\{z = x+iy: |x|< \beta-\delta,~ y\in \mathbb{R}\}$
is at distance at least $\delta$ from every zero of $\det \Delta (z)$, we have $$|\det \Delta(z)|\geq m(\delta).$$
\end{lem}
For the sake of completeness, we recall the proof 
of Lemma~\ref{lem:lewin_perturbed} (i.e. \cite[Proposition A.5]{chitour:hal-05337430}) in Appendix~\ref{proof:lewin}.
We now state and prove the stability criterion for the IDE~\eqref{main IDE}.
\begin{thm}[Exponential Stability Criterion]
    \label{stable lemma}
 The IDE~\eqref{main IDE} is exponentially stable if and only if there exists $\nu_0>0$ such that 
 all solutions $z\in \mathbb{C}$ of the characteristic equation~\eqref{Characteristic eq} satisfy $\Re(z)<-\nu_0.$ Moreover, in that case,
 for all $0<\nu<\nu_0$, there exists $C\geq1$ such that for all $X_0\in V((-\tau_*, 0), \R^n)$, the solution to~\eqref{main IDE} satisfies, for all $t\geq 0$,
    $$\|X_t\|_V\leq Ce^{-\nu t}\|X_0\|_V.$$
\end{thm}
\begin{proof}
Suppose the IDE~\eqref{main IDE} is exponentially stable, i.e. there exists $\nu >0$, $C\geq1$ such that for all $X_0\in V((-\tau_*, 0), \R^n)$, the solution to~\eqref{main IDE} satisfies, for all $t\geq 0$,
    $$\|X_t\|_V\leq Ce^{-\nu t}\|X_0\|_V.$$ Now suppose by contradiction that  
    for all $\nu_0>0$, there exists $z_0\in\mathbb{C}$ solution of the characteristic equation~\eqref{Characteristic eq} such that $\Re(z_0)+\nu_0 \geq0$.
    Set $\nu_0 = \nu/2$. Then, \begin{equation}\label{ineq:nu_plus_z0}\Re(z_0)+\nu >0.\end{equation}
    Then, the conjuguate $\Bar{z}_0$ is also a solution of the characteristic 
    equation~\eqref{Characteristic eq}.
Let  $v_0\in \mathbb{C}^n$ be a corresponding eigenvector, i.e. $v_0\neq0$ and,
\[
\Big(I + \sum_{k=1}^\infty A_k e^{-z_0 \tau_k} + \int_0^{\tau_*} N(s) e^{-z_0 s} ds \Big) v_0 = 0.
\]
Then the real vector-valued function
\[
X(t) := \frac{1}{2} \Big( e^{z_0 t} v_0 + e^{\bar{z}_0 t} \bar{v}_0 \Big) = \Re\big(e^{z_0 t} v_0\big),
\]
is a solution of the IDE~\eqref{main IDE}.
For all $t\geq 0$, for all $s\in [-\tau_*, 0]$,
$$X_t(s) = \Re(e^{z_0(t+s)}v_0) = e^{\Re(z_0)t}\Re(e^{i\Im(z_0)t+z_0s}v_0).$$
Therefore, for all $t\geq 0$,
$$\|X_t\|_V =e^{\Re(z_0)t}h_0(t),$$
with 
$$h_0(t):=\|s\mapsto \Re(e^{i\Im(z_0)t+z_0s}v_0)\|_V.$$
If $\Im(z_0)= 0$, then $h_0$ is constant. Otherwise, $h_0$ is a periodic continuous function with period $T_0:= 2\pi/|\Im(z_0)|$. In all cases, $h_0$ attains its maximum, denoted $M_0$, at $t_0\geq 0$. Since $v_0\neq 0$, we have $M_0>0$.
Denotes for all $k\in \N$, 
$t_k:= t_0 + kT_0$. Then,
$$h_0(t_k) = M_0, \quad \lim_{k\to +\infty}t_k = +\infty.$$
Hence, using the exponential stability of the IDE~\eqref{main IDE}, for all $k\in \N$,
$$e^{\Re(z_0)t_k}h_0(t_k)=\|X_{t_k}\|_V \leq Ce^{-\nu t_k}\|X_0\|_V,$$
i.e
$$e^{(\Re(z_0)+\nu)t_k}\leq \frac{C}{M_0}\|X_0\|_V.$$
Due to~\eqref{ineq:nu_plus_z0}, $\Re(z_0)+\nu>0$. Therefore, taking $k$ large enough leads us the contradiction.
\medskip

    Let us show the converse.
    Using Proposition~\ref{one to one}, let $x$ be the unique solution to~\eqref{IDE measure form} with $f$ given by~\eqref{def:f_IDE_measure}.
    Then, for all $t\geq -\tau_*$, 
\begin{equation}\label{1t1}X(t)= x(t)\mathbbb{1}_{t\geq 0}(t) + X_0(t)\mathbbb{1}_{-\tau_*\leq t<0}(t).\end{equation}
    Let $\nu>0$ be such that $\delta:=\nu_0-\nu >0$, and define $\eta(s) = e^{\nu s}$. The weight function $\eta$ is a self-dominated sub-multiplicative function (see~\cite[Definition 3.2, Section 4.2]{GripenbergLondenStaffans1990}).
   To establish the exponential stability of the solution \(X\) to~\eqref{main IDE}, we first show that the corresponding solution \(x\) of~\eqref{IDE measure form} belongs to \(V(\mathbb{R}_+, \eta, \mathbb{R}^n)\).  
This follows from the existence of a constant \(C > 0\) such that  
\begin{equation} \label{ineg x}
    \|x\|_{V_\eta(\R_+)} \leq C \|X_0\|_{V(-\tau_*,0)}.
\end{equation}
    We will use the formula for the solution $x$ of equation~\eqref{IDE measure form} given in Proposition~\ref{one to one}. Let us check the requirements to apply~\cite[Corollary 4.7, Section 4.4]{GripenbergLondenStaffans1990}. The measure $\mu$ has compact support in $[0, \tau_*]$, hence  $\mu$ is in $M(\mathbb{R}_+, \eta, \mathbb{R}^{n\times n})$ and $M(\mathbb{R}, \eta, \mathbb{R}^{n\times n})$. 
     The measure $\mu$ is in the form of~\cite[Equation (2.1) Section 3.2]{GripenbergLondenStaffans1990} with no singular part.
   Since, by dominated convergence, 
    $$\lim _{\Re(z)\to +\infty} |R(z)| = 0,$$
    we have by continuity of the determinant,
    $$\lim _{\Re(z)\to +\infty} \det \Delta(z)=\lim _{\Re(z)\to +\infty} \det \Delta_0(z) =\det I = 1\neq0. $$
    Hence, using Lemma~\eqref{lem:lewin_perturbed},
\begin{equation}\inf_{\Re(z) \geq -\nu}|\det (I+\hat{\mu}(z))|>0.\label{eq:infhatmu}\end{equation}
Then, we can apply~\cite[Corollary 4.7, Section 4.4]{GripenbergLondenStaffans1990}. Therefore, the resolvent $\rho$ is in $M(\mathbb{R}_+, \eta)$.
Using the explicit formula of the solution $x$ of equation~\eqref{IDE measure form} given in Proposition~\ref{one to one}, we obtain applying~\cite[Theorem 3.5, Section 4.3]{GripenbergLondenStaffans1990},
\begin{align*}
     \|x\|_{V_{\eta}(\R_+)}&= \|f-\rho*f\|_{V_{\eta}(\R_+)}\\
     &\leq (1+\|\rho\|_{M(\mathbb{R}_+,\eta)})\|f\|_{V_{\eta}(\R_+)}\\
     &\leq (1+\|\rho\|_{M(\mathbb{R}_+,\eta)})\|\mu*X_0\|_{V_{\eta}(\R)}\\
     &\leq (1+\|\rho\|_{M(\mathbb{R}_+,\eta)})\|\mu\|_{M(\mathbb{R},\eta)}\|X_0\|_{V_{\eta}(\R)}.
\end{align*}
Because, 
$$\|X_0\|_{V_{\eta}(\R)} =\|\eta X_0\|_{V(-\tau_*,0)} \leq \sup_{s\in[-\tau_*, 0]}e^{\nu s }\|X_0\|_{V(-\tau_*,0)} \leq \|X_0\|_{V(-\tau_*,0)},$$
the inequality~\eqref{ineg x} is proved.
Moreover, using~\eqref{1t1}, we obtain
for all $t\geq0$
\begin{align*}
    \|X_t\|_{V(-\tau_*,0)}&=\|X\|_{V(t-\tau_*, t)}\\
    &\leq \sup_{s\in [t-\tau_*, t]}e^{-\nu s}\|\eta X\|_{V(t-\tau_*, t)}\\
    &\leq e^{-\nu (t-\tau_*)}(\|x\|_{V_\eta(\R_+)}+\|X_0\|_{V_\eta(\R)})\\
    &\leq Ce^{-\nu t}\|X_0\|_{V(-\tau_*,0)},
\end{align*}
with $C = e^{\nu \tau_*}\big((1+\|\rho\|_{M(\mathbb{R}_+,\eta)})\|\mu\|_{M(\mathbb{R},\eta)}+1\big).$
\end{proof}
\section{Well-posedness and an exponential stability criterion for a DDE}
\label{critere_DDE}
In this section, we explain how the approach developed in~\cite{GripenbergLondenStaffans1990} can be used to state the same exponential stability criterion (Theorem~\ref{stable lemma}) but for a DDE.
 From now on, for $C\subset \R$ with nonempty interior, $V(C, \R^n)$ is a Banach space from the list:
    \begin{enumerate}
        \item $L^p(C, \R^n)$, for $p\in [1,+\infty]$; 
        \item $BV(C, \R^n)$.
    \end{enumerate}  
The space $B^\infty$ is not considered here, as the results of~\cite{GripenbergLondenStaffans1990} do not allow one to draw conclusions in this setting.
We consider the following DDE, for any $(X_0, x_0) \in V([-\tau_*,0], \R^n)\times \R^n$,
\begin{equation}\label{main DDE}
\begin{cases}
 X'(t) + \sum_{k=1}^{\infty} A_k X(t-\tau_k) + 
 \int_0^{\tau_*} N(s) X(t-s)\,ds = 0,
 \quad \textit{ $t>0$},\\[4pt]
 X(t) = X_0(t), \quad \textit{$-\tau_* \le t < 0 $
 },\\[4pt]
 X(0) =x_0
\end{cases}
\end{equation}
A function $X$ is a solution of~\eqref{main DDE} if, $X$ is locally absolutely continuous, for all $t > 0$, we have $X'_t \in V([-\tau_*,0],\mathbb{R}^n)$, and the following conditions hold:  
when $V = B^\infty$, equation~\eqref{main DDE} must hold for all $t \ge -\tau_*$;  
when $V = BV$ or $V = L^p$, equation~\eqref{main IDE} is required to hold for almost every $t$.
In the case of $V$ being a Lebesgue space, one could also consider $M^p$ spaces for the initial data of~\eqref{main DDE} (see e.g.~\cite{DelfourMitter1972HereditaryI}).
Again, we can rewrite~\eqref{main DDE} in the formalism of~\cite{GripenbergLondenStaffans1990}, the well-posedness is then a consequence of a result from the same reference.
\begin{prop} \label{one to one:DDE}
For all  initial data $(X_0,x_0)\in V([-\tau_*, 0], \R^n)\times \R^n$, if $X$ is a solution of~\eqref{main DDE} then $x(t):=X(t)\mathbbb{1}_{t\geq 0}(t)$ is a solution of
    \begin{equation}
        \label{DDE measure form}
    x'(t) +\int_0^t \mu(ds)x(t-s) = f(t),\quad x(0)=x_0, \quad t\geq 0,
    \end{equation}
where $\mu$ and $f$ are given in Proposition~\ref{one to one}.
Conversely, if $x$ is a solution to~\eqref{IDE measure form}, then $$X(t)= x(t)\mathbbb{1}_{t> 0}(t) + X_0(t)\mathbbb{1}_{-\tau_*\leq t<0}(t),$$
with $X(0)=x_0$,
is a solution to~\eqref{main IDE}.
Moreover, there is a unique locally absolutely continuous solution $x$ to equation~\eqref{IDE measure form} with $x' \in V_{\textit{loc}}(\mathbb{R}_+, \mathbb{R}^{n})$. The solution is given by 
\begin{equation}
    \label{sol explicit:neutral}
    x(t) = r(t)x_0 +(r*f)(t), \quad t\geq 0,
\end{equation}
where $r$ is the differential resolvent of $\mu$, i.e. the unique solution of $$r' +\mu*r = r' +r*\mu = 0, \quad r(0)=I,$$
such that $r$ is locally absolutely continuous on $\R^+$. Furthermore, $r'$ is equal almost everywhere to a function in $BV_{\textit{loc}}(\R^+, \R^{n\times n}).$
\end{prop}
\begin{proof}
    The equivalence between the equations~\eqref{main DDE} and~\eqref{DDE measure form} is proved the same way as in the IDE case (see Proposition~\ref{one to one}). The rest of the proof follows from a direct application of~\cite[Theorem 3.1 and 3.2, Section 3.3]{GripenbergLondenStaffans1990}.
\end{proof}
We give the definition of exponential stability for the DDE~\eqref{main DDE}.
\begin{defn}[Exponential stability]
 The DDE~\eqref{main DDE} is said to be exponentially stable, if there exist $\nu_0 >0$, 
 $C\geq1$ such that for all $(X_0,x_0)\in V((-\tau_*, 0), \R^n)\times \R^n$, the solution to~\eqref{main DDE} satisfies, for all $t\geq 0$,
    $$\|X_t\|_V + |X(t)| \leq Ce^{-\nu_0 t}(\|X_0\|_V+|x_0|).$$
\end{defn}
We define the characteristic equation associated for the DDE~\eqref{main DDE}.
\begin{defn}[Characteristic equation]
    
    The characteristic equation associated to the DDE~\eqref{main DDE} is defined for all $z\in \mathbb{C}$, by
    \begin{equation}\label{Characteristic eq:DDE}
       \operatorname{det}\Delta(z) = 0,
    \end{equation}
    where 
    $$\Delta(z) :=zI+\hat{\mu}(z).$$
\end{defn}
We end this section by stating the exponential stability criterion of the DDE~\eqref{main DDE}.
\begin{thm}[Exponential Stability Criterion]
    \label{stable lemma:DDE}
 The DDE~\eqref{main DDE} is exponentially stable if and only if there exists $\nu_0>0$ such that 
 all solutions $z\in \mathbb{C}$ of the characteristic equation~\eqref{Characteristic eq:DDE} satisfy $\Re(z)<-\nu_0.$ Moreover, in that case, there exists $C\geq1$ such that for all $(X_0,x_0)\in V((-\tau_*, 0), \R^n)\times \R^n$, the solution to~\eqref{main DDE} satisfies, for all $t\geq 0$,
    $$\|X_t\|_V + |X(t)|\leq Ce^{-\nu_0 t}(\|X_0\|_V+|x_0|).$$
\end{thm}
\begin{rem}
In Theorem~\eqref{stable lemma:DDE}, in contrast with Theorem~\ref{stable lemma}, no margin $\nu < \nu_0$ is required. 
This is a direct consequence of the fact that, for the DDE~\eqref{main DDE},~\cite[Theorem 4.13, Section 4.4]{GripenbergLondenStaffans1990} applies without the need to verify a condition like the one proved in Lemma~\eqref{lem:lewin_perturbed}.
\end{rem}
\begin{proof}
    The direct implication is proved exactly like in the IDE case.
    For the other direction, we define the same exponential weight $\eta(s):= e^{\nu_0 s}$. 
    Using Proposition~\ref{one to one:DDE}, let $x$ be the unique solution to~\eqref{DDE measure form} with $f$ given by~\eqref{def:f_IDE_measure}.
    Then, for all $t\geq -\tau_*$, we have
\begin{equation}\label{1t1:DDE}X(t)= x(t)\mathbbb{1}_{t\geq 0}(t) + X_0(t)\mathbbb{1}_{-\tau_*\leq t<0}(t),\quad x(0) =x_0.\end{equation}
Let us denote $r$ the differential resolvent of $\mu$. Applying~\cite[Theorem 4.13, Section 4.4]{GripenbergLondenStaffans1990}, we obtain $r \in L^1(\R^+, \eta, \R^{n\times n}).$
    Moreover, $r'\in L^1(\R^+, \eta, \R^{n\times n})$, and $r'$ is almost everywhere equal to a function in $BV(\R^+, \eta, \R^{n\times n})$.
Hence, $r\in L^\infty(\R^+, \eta, \R^{n\times n})\cap L^1(\R^+, \eta, \R^{n\times n})$ using the fundamental theorem of calculus; therefore, $r\in L^p(\R^+, \eta, \R^{n\times n})$ for all $p\in [1,+\infty]$. Moreover, 
it is a classical fact that if $r'\in L^1(\R_+, \eta, \R^{n\times})$, then the total variation of $r\in L^1(\R_+, \eta, \R^{n\times})$ is smaller or equal to the norm of $r'$ in $ L^1(\R_+, \eta, \R^{n\times})$; hence $r\in BV(\R^+, \eta, \R^{n\times n})$.

We now use the explicit formula for the solution $x$ of equation~\eqref{DDE measure form} given in Proposition~\ref{one to one:DDE}, Young's convolution inequality, and~\cite[Theorem 3.5, Section 4.3]{GripenbergLondenStaffans1990} to bound $\|x\|_{V_\eta(\R_+)}$ as follows.
In the case $V= L^p$, we have  
    \begin{align*}
     \|x\|_{L^p_{\eta}(\R_+)}&= \|rx_0+r*f\|_{L^p_{\eta}(\R_+)}\\
     &\leq \|r\|_{L^p_\eta(\R_+)}|x_0|+\|r\eta *f\eta\|_{L^p(\R_+)}
     \\
     &\leq \|r\|_{L^p_\eta(\R_+)}|x_0|+\|r\|_{L^1_\eta(\R_+)}\|f\|_{L^p_\eta(\R_+)}\\
     &\leq \|r\|_{L^p_\eta(\R_+)}|x_0| +\|r\|_{L^1_\eta(\R_+)} \|\mu\|_{M(\R,\eta)}\|X_0\|_{L^p_\eta(\R)}.
\end{align*}
In the case $V = BV$, we get
  \begin{align*}
     \|x\|_{BV_{\eta}(\R_+)}&= \|rx_0+r*f\|_{BV_{\eta}(\R_+)}
     \\
     &\leq \|r\|_{BV_\eta(\R_+)}|x_0|+\|r'*f\|_{L^1_\eta(\R_+)}\\
     &\leq \|r\|_{BV_\eta(\R_+)}|x_0|+\|r'\eta*f\eta\|_{L^1(\R_+)}\\
       &\leq \|r\|_{BV_\eta(\R_+)}|x_0|+\|r'\|_{L^1_\eta(\R_+)}\|f\|_{L^1_\eta(\R_+)}\\
       &\leq \|r\|_{BV_\eta(\R_+)}|x_0|+\|r'\|_{L^1_\eta(\R_+)}\tau_*\|f\|_{BV_\eta(\R_+)}\\
     &\leq \|r\|_{BV_\eta(\R_+)}|x_0| +\|r'\|_{L^1_\eta(\R_+)} \tau_*\|\mu\|_{M(\R,\eta)}\|X_0\|_{BV_\eta(\R)},
\end{align*}
where we have used the fact that $f$ is compactly supported in $[0, \tau_*]$, which implies
$$\|f\|_{L^1_\eta(\R_+)}\leq \tau_*\|f\|_{L^\infty_\eta(\R_+)}\leq \tau_*\|f\|_{BV_\eta(\R_+)}.$$
Furthermore, we have
$$\|X_0\|_{V_{\eta}(\R)} =\|\eta X_0\|_{V(-\tau_*,0)} \leq \sup_{s\in[-\tau_*, 0]}e^{\nu_0 s }\|X_0\|_{V(-\tau_*,0)} \leq \|X_0\|_{V(-\tau_*,0)}.$$
Hence, using~\eqref{1t1:DDE}, there exists $c>0$, and $K>0$ such that, 
for all $t\geq0$,
\begin{align*}
    \|X_t\|_{V(-\tau_*,0)}&=\|X\|_{V(t-\tau_*, t)}\\
    &\leq \sup_{s\in [t-\tau_*, t]}e^{-\nu_0 s}\|\eta X\|_{V(t-\tau_*, t)}\\
    &\leq e^{-\nu_0 (t-\tau_*)}(\|x\|_{V_\eta(\R_+)}+\|X_0\|_{V_\eta(\R)})\\
    &\leq  e^{-\nu_0 (t-\tau_*)}( c|x_0| + (K+1)\|X_0\|_{V(-\tau_*,0)}).
\end{align*}
On the other hand, for all $t\geq 0$,
\begin{align}
    |X(t)| &= |x(t)|\nonumber\\
    &\leq |r(t)x_0| + |(r*f)(t)| \nonumber\\
    &= e^{-\nu_0 t}\big(|\eta(t) r(t)||x_0| +|(r\eta*f\eta)(t)|\big) \nonumber\\
    &\leq e^{-\nu_0 t}\sup_{t \in \R^+}|\eta(t) r(t)|\big(|x_0| +\|f\eta\|_{L^1}\big) \tag{$\sup_{t \in \R^+}|\eta(t) r(t)|<+\infty$}\\
    &\leq e^{-\nu_0 t}\sup_{t \in \R^+}|\eta(t) r(t)|\big(|x_0| +\|\mu *X_0\|_{L^1_\eta}\big) \nonumber\\
    &\leq e^{-\nu_0 t}\sup_{t \in \R^+}|\eta(t) r(t)|\big(|x_0| +\|\mu\|_{M(\R, \eta)}\|X_0\|_{L^1_\eta}\big) \tag{from \cite[Theorem 3.5, Section 4.3]{GripenbergLondenStaffans1990}}.
\end{align} 
Finally, if $V=L^p$, since $X_0$ is compactly supported, Hölder inequality yields $\|X_0\|_{L^1_\eta}\leq C\|X_0\|_{V}$ for some $C\geq0$. If $V=BV$, then $\|X_0\|_{L^1_\eta}\leq C\|X_0\|_{L^\infty}\leq C\|X_0\|_{BV}$.
Thus, in any case,
\begin{equation}
    |X(t)| \leq e^{-\nu_0 t}\sup_{t \in \R^+}|\eta(t) r(t)|\big(|x_0| +C\|\mu\|_{M(\R, \eta)}\|X_0\|_{V}\big).
    \label{ligne_finale}
\end{equation}
This concludes the proof.
\end{proof}
\section{Conclusion}
In this article, we have established a spectral criterion for the exponential stability of IDEs and DDEs in several state-space settings, relying on the framework developed in~\cite[Chapter~4]{GripenbergLondenStaffans1990}. 
Whereas this criterion is classical in the standard space of continuous functions, our contribution lies in extending its validity to significantly broader functional spaces. 
This extension is crucial for the PDE applications discussed in the introduction, where Lebesgue spaces arise naturally from the underlying modeling.
\section*{Acknowledgment}
The authors would like to thank Guilherme Mazanti for his valuable comments and suggestions during the preparation of this manuscript.
\appendix
\section{Proof of Lemma~\ref{lem:lewin_perturbed}} 
\label{proof:lewin}
We first establish the following auxiliary proposition.
\begin{prop} For all $\beta>0$, $$\lim_{|y|\to \infty} \sup_{x\in [-\beta, \beta]}|R(x+iy)| = 0.$$ \label{prop:rm}
\end{prop}
\begin{proof}

For each \(y \in \R\), define
\[
R_y(x) := |R(x+iy)|, \qquad x \in [-\beta, \beta].
\]
Let \((y_n)_n\) be a sequence such that \(|y_n| \to \infty\).
For every \(n \in \N\), the function \(R_{y_n}\) is continuous on the compact interval \([-\beta,\beta]\);
hence there exists \(x_n \in [-\beta,\beta]\) such that
\[
\sup_{x\in[-\beta,\beta]} R_{y_n}(x) = R_{y_n}(x_n).
\]
Up to the extraction of a subsequence, we may assume \(x_n \to x_0 \in [-\beta,\beta]\).

Since \(N \in L^1([0,\tau_*];\R^{m\times m})\) and
\(|e^{x_n s} - e^{x_0 s}| \le 2e^{\beta\tau_*}\) for all \(s \in [0,\tau_*]\),
the dominated convergence theorem implies
\[
|R(x_n+iy_n) - R(x_0+iy_n)|
\le \int_0^{\tau_*} |N(s)|\,|e^{x_n s} - e^{x_0 s}|\,ds
\longrightarrow 0 \quad \text{as } n \to \infty.
\]
Moreover, by the Riemann--Lebesgue lemma, we have
\[
|R(x_0+iy_n)|
= \Big|\int_0^{\tau_*} N(s)e^{x_0 s} e^{i y_n s}\,ds\Big|
\longrightarrow 0.
\]
Combining the two limits yields \(R(x_n + i y_n) \to 0\).
Therefore,
\[
\sup_{x\in[-\beta,\beta]} |R(x+i y_n)| = R_{y_n}(x_n) \longrightarrow 0.
\]
Since this holds for any sequence \((y_n)\) with \(|y_n| \to \infty\), we conclude that
\[
\displaystyle
\lim_{|y|\to\infty} \sup_{x\in[-\beta,\beta]} |R(x+iy)| = 0.
\]
\end{proof}
We now give below the proof of Lemma~\ref{lem:lewin_perturbed}.
\begin{proof}[Proof of Lemma~\ref{lem:lewin_perturbed}] 
    The proof is an adaption of~\cite[Lemma 1, Section VI.2 page 268]{Levin1972}.
    Suppose by contradiction
    that there exists $\beta >0$ and $0<\delta<\beta$ such that there exists $(z_j)_{j\in \N}$ such that for all $j\in \N$, $z_j = x_j +iy_j \in \Omega_{\delta}$ with $|y_j|\xrightarrow[j\to\infty]{}+\infty$, and $\operatorname{dist}(z_j, Z)>\delta$, where $Z:= \{z\in \C: \det \Delta(z) = 0\}$,
    and \begin{equation} \label{lim:delta_0}|\det \Delta(z_j)| \xrightarrow[j\to\infty]{} 0.
\end{equation}
Since $|x_j|< \beta-\delta$ for all $j\in \N$, by Bolzano-Weirstrass theorem, we can extract a subsquence (still denoted by $x_j$) such that $x_j\xrightarrow[j\to\infty]{} x_0\in (-\beta, \beta)$.
Moreover, up to a diagonal extraction of subsequences, we can assume that for all $k\in \N^*$, there exists $\omega_k$ such that $e^{-i\tau_k y_j}\underset{j\to +\infty}{\longrightarrow}e^{i\omega_k}.$
This is a consequence of the compactness of the range of the periodic function $y\mapsto e^{iy}$ defined for all $y\in\R$.
Define \fonction{\Delta_0^\omega}{[-\beta,\beta]}{ \C^{n\times n}}{x}{I + \sum_{k=1}^\infty A_k e^{-\tau_k x}e^{i\omega_k}.}
Hence, for all $x\in [-\beta, \beta]$,
$$|\Delta_0(x+iy_j) - \Delta_0^\omega(x)| \xrightarrow[j\to\infty]{} 0.$$
For all $j\in \N$, by continuity on a compact set, we can denote $x_j^m \in [-\beta,\beta]$ such that:
$$\sup_{x\in [-\beta, \beta]}|\Delta_0(x+iy_j) - \Delta_0^\omega(x)| = |\Delta_0(x_j^m+iy_j) - \Delta_0^\omega(x_j)|.$$ 
From there we have, 
$$|\Delta_0(x_j^m+iy_j) - \Delta_0^\omega(x_j)| = |\sum_{k=0}^{\infty} A_ke^{-\tau_k x_j^m}(e^{-i\tau_ky_j}-e^{i\omega_k})|\xrightarrow[j\to\infty]{} 0, $$
and by dominated convergence,
$$|\sum_{k=0}^{\infty} A_ke^{-\tau_k x_j^m}(e^{-i\tau_ky_j}-e^{i\omega_k})|\leq 2e^{\tau_*\beta}\sum_{k=0}^ \infty |A_k|<+\infty.$$
Therefore, $$\sup_{x\in [-\beta, \beta]}|\Delta_0(x+iy_j) - \Delta_0^\omega(x)| \xrightarrow[j\to\infty]{} 0.$$
Then, using Proposition~\ref{prop:rm},
$$\sup_{x\in [-\beta, \beta]}|\Delta(x+iy_j) - \Delta_0^\omega(x)| \xrightarrow[j\to\infty]{} 0.$$
Furthermore, using the continuity of the determinant, we obtain,
\begin{equation}\sup_{x\in [-\beta, \beta]}|\det\Delta(x+iy_j) - \det \Delta_0^\omega(x)| \xrightarrow[j\to\infty]{} 0.\label{lim:absurd_delta}\end{equation}
Equations~\eqref{lim:absurd_delta} and~\eqref{lim:delta_0} together imply that 
\[
\det \Delta_0^\omega(x_{0})  = 0.
\]

We want to apply Hurwitz theorem~\cite[§5.3.4]{Krantz1999Handbook}.
Observe that $\Delta_0^\omega$ cannot vanish identically on $[-\beta,\beta]$.  Indeed, if $\Delta_0^\omega\equiv 0$ on this interval then, by the identity theorem for holomorphic functions~\cite[Theorem 2, Section 1.3]{NarasimhanNievergelt2001Complex}, the natural extension of $\det\Delta_0^\omega$ (denoted the same) on $\C$ would vanish identically.  This is impossible, since
\[
\lim_{\Re(z) \to +\infty} \det\Delta_0^\omega(z)=\lim_{\Re(z) \to +\infty}\det\left(I+\sum_{k=1}^\infty A_ke^{iw_k}e^{-\tau_k z}\right)=\det I=1.
\]

From Hurwitz theorem on the roots
of uniformly convergent sequences of holomorphic functions (see e.g.~\cite[§5.3.4]{Krantz1999Handbook}), as $\det \Delta_0^\omega(y_0)$ is not identically equal to zero, there exists $J>0$ such that for all $j>J$, $\det \Delta(\cdot +iy_j)$ has a zero in a $\delta-$neighborhood of $z_0:=x_0+iy_0$. Hence, $\det \Delta$ has infinitely many zeros at distance smaller than $\delta$ from a $z_j$. This is a contradiction because the $z_j$ are at distance $\delta$ from every zeros of $\det \Delta$.
\end{proof}
\bibliographystyle{abbrv}
\bibliography{references}
\end{document}